\documentclass[11pt]{article}

\usepackage{amsmath}
\usepackage{a4}
\usepackage[english]{babel}
\usepackage{enumerate}

\usepackage{amsthm}

\newtheorem{lemma}{Lemma}[section]

\newtheorem{corollary}[lemma]{Corollary}

\newtheorem{theorem}[lemma]{Theorem}

\newtheorem{remark}[lemma]{Remark}

\usepackage{amsmath}
\usepackage{amssymb}
\usepackage{graphicx}
\usepackage[T1]{fontenc}
\usepackage[latin1]{inputenc}
\usepackage{ae}
\usepackage{pictexwd,dcpic}

\input amssym.def

\begin{document}
%\begin{center}
%\textbf{Partially isometric multipliers}
%\vspace{1cm}\\
%\end{center}
\setcounter{section}{-1}

\section*{
\begin{center}
Wandering subspace property for homogeneous\\ invariant subspaces
\end{center}
}
\begin{center}
J\"org Eschmeier
\end{center}
%\begin{center}
%\small{
%This paper is dedicated to the memory of Professor Ronald G. Douglas}
%\end{center}

\vspace{.8cm}

\begin{center}
\parbox{12cm}{\small 
\vspace{0.5cm}

For graded Hilbert spaces $H$ and shift-like
commuting tuples $T \in B(H)^n$, we show that each homogeneous joint invariant subspace $M$ of $T$ has finite index and is
generated by its wandering subspace. Under suitable conditions on the grading $(H_k)_{k\geq 0}$ of $H$ 
the algebraic direct sum $\tilde{M} = \oplus_{k\geq 0} M \cap H_k$ becomes a finitely generated module over
the polynomial ring $\mathbb C[z]$. We show that the wandering subspace $W_T(M)$ of $M$ is contained in 
$\tilde{M}$ and that each linear basis of $W_T(M)$ forms a minimal set of generators for
$\tilde{M}$. We describe an algorithm that transforms each set of homogeneous generators of $\tilde{M}$ into
a minimal set of generators and can be used to compute minimal sets of generators for homogeneous
ideals $I \subset \mathbb C[z]$. We prove that each finitely generated $\gamma$-graded commuting row contraction
$T \in B(H)^n$ admits a finite weak resolution in the sense of Arveson \cite{A} or Douglas and Misra \cite{DM}.

\vspace{.5cm}

\emph{2010 Mathematics Subject Classification:} 47A13, 13P10\\
\emph{Key words and phrases:} Wandering subspaces, homogeneous invariant subspaces, minimal sets of generators,
weak resolutions}

\end{center}

\section{Introduction}
Let $T \in B(H)$ be a bounded linear operator on a complex Hilbert space $H$ and let $M \subset H$ be a closed $T$-invariant subspace.
Following Halmos \cite{H} one defines its \textit{wandering subspace} by $W_T(M) = M\ominus TM$. 
It is obvious from the definition that
 \[
W_T(M) \perp T^k W_T(M)
\]
for all $k \in \mathbb N \setminus \{ 0 \}$.  Beurling's invariant subspace theorem implies that, for each
closed linear subspace $M \subset H^2(\mathbb T)$ of the Hardy space on the unit circle which is invariant
under the operator $M_z$ of multiplication with the argument, the wandering subspace $W_{M_z}(M)$ is
one-dimensional and generates $M$ as an invariant subspace. More precisely, there is an inner function
$\theta \in H^{\infty}(\mathbb T)$ such that $W_{M_z}(M) = \mathbb C \theta$ and
\[
M = \bigvee_{k \in \mathbb N} M^k_z W_{M_z}(M) = \theta H^2(\mathbb T).
\]
For an arbitrary operator $T \in B(H)$ and a closed $T$-invariant subspace $M \subset H$, the number 
${\rm ind}(T) = \dim W_T(M)$ is usually called the \textit{index} of $M$. The operator $T$ is said to possess the
\textit{wandering subspace property} if each closed $T$-invariant subspace $M \subset H$ is generated by its
wandering subspace in the sense that
\[
M = \bigvee _{k \in \mathbb N} T^k W_{T}(M).
\]
It is well known that the Bergman space $L^2_a(\mathbb D)$ of all square integrable analytic functions on the unit disc
contains $M_z$-invariant subspaces of infinite index (cf. \cite{ABFP} or \cite{HRS}).
 On the other hand, it was shown by Aleman, Richter and Sundberg \cite{ARS}
that $M_z \in B(L^2_a(\mathbb D))$ possesses the wandering subspace property. More generally, results of Shimorin \cite{Sh},
Hedenmalm and Zhu \cite{HZ}, and others, show that the multiplication operator $M_z$ on the analytic functional Hilbert spaces
$H(K_m) \subset \mathcal O(\mathbb D)$ given by the reproducing kernels
\[
K_m(z,w) = \frac{1}{(1 - z\overline{w})^m}
\]
possesses the wandering subspace property for $1 \leq m \leq 3$, but fail to have this property for $m > 6$.

If the unit disc is replaced by the unit ball in $\mathbb B \subset \mathbb C^n$, then the situation becomes more
involved. 
%It turns out that many classical results known for the Hardy space on the unit circle do not 
%extend to the Hardy space on the unit sphere in $\mathbb C^n$, but that there are versions of these results for
%he Drury-Arveson space, that is, the analytic functional Hilbert space 
%$H(\mathbb B) \subset \mathbb O(\mathbb B)$ given by the reproducing kernel 
%\[
%K(z,w) = \frac{1}{1 - \langle z,w \rangle  \quad \quad (z, w \in mathbb B).
%\]
It is known that in the standard analytic functional Hilbert spaces on the unit ball such as the 
Drury-Arveson space, the Hardy or Bergman space, there exist
invariant subspaces of infinite index and that each of these spaces contains invariant subspaces that
are not generated by their wandering subspace \cite{BEKS}.\\

In this note we show that, for Hilbert spaces $H$ that admit a suitable grading 
$H = \bigcirc\!\!\!\!\!\!\!\!\perp ^\infty_{k=0} H_k$ and shift-like commuting tuples $T \in B(H)^n$, each
homogeneous joint invariant subspace $M$ of $T$ has finite index and is generated by its wandering subspace.
More precisely, let us suppose that there is a weight tuple $\gamma = (\gamma_1, \ldots, \gamma_n)$ consisting of positive
integers $\gamma_i > 0$ such that
\[
T_i H_k \subset H_{k + \gamma_i} \quad \mbox{ for } k \in \mathbb N  \mbox{ and } i = 1, \ldots, n.
\]
We call a closed subspace $M \subset H$ \textit{homogeneous} if it admits the orthogonal decomposition
$M = \bigcirc\!\!\!\!\!\!\!\!\perp ^\infty_{k=0} M \cap H_k$. 
For a closed joint $T$-invariant subspace $M \subset H$, its \textit{wandering subpace}
is defined as $W_T(M) = M \ominus \sum^n_{i=1} T_i M$. Under the hypothesis that the space $H_0$ is finite
dimensional and that $H = \bigvee_{k \in \mathbb N^n} T^k H_0$ we show that each
wandering subspace of a homogeneous joint $T$-invariant subspace $M$ is finite dimensional and that $M$ is generated
by its wandering subspace in the sense that
\[
M = \bigvee_{k \in \mathbb N^n} T^k W_T(M).
\]
Furthermore we show that the maps assigning to each homogeneous joint $T$-invariant subspace its wandering subspace
and conversely, to each homogenous wandering subspace of $H$ the invariant subspace that it generates, define
bijections that are inverse to each other.\\

Under the above conditions the algebraic direct sum $\tilde{M} = \oplus ^\infty_{k=0} M \cap H_k$ becomes a finitely 
generated $\mathbb C[z]$-submodule of $H$ with respect to the module structure
\[
\mathbb C[z] \times H \rightarrow H, (p,T) \mapsto p(T)x.
\]
We show that, for each homogeneous joint $T$-invariant subpace $M \subset H$, its wandering subspace $W_T(M)$ is contained
in $\tilde{M}$ and that each linear basis of $W_T(M)$ forms  a minimal set of generators for $\mathbb C[z]$-module
$\tilde{M}$. We deduce an algorithm which transforms each finite set of homogeneous generators for $\tilde{M}$ into
an orthogonal basis of $W_T(M)$ and hence into a minimal set of generators for $\tilde{M}$. By applying the above
results to the closure of a homogenous ideal $I \subset \mathbb C[z]$ in the Hardy space on the unit polydisc
$\mathbb D^n \subset \mathbb C^n$ one obtains an algorithm that transforms each set of homogeneous generators of
$I$ into a minimal set of homogeneous generators. We prove that each finitely generated $\gamma$-graded
commuting row contraction $T \in B(H)^n$ admits a unique weak resolution in the sense of Arveson \cite{A} and
Douglas and Misra \cite{DM}. We thus extend results of Arveson \cite{A} to the case of $\gamma$-graded tuples.

\section{Preliminaries}
Let $T=(T_1,\ldots,T_n)\in B(H)^n$ be a commuting tuple of bounded linear operators on a complex Hilbert space $H$. We write ${\rm Lat}(T)$ for the set of all closed linear
subspaces $M \subset H$ such that $T_iM\subset M$ for $i=1,\ldots,n$. For an arbitrary subset $M \subset H$, we denote by
\[
[M] = \bigvee (T^k M;\ k \in \mathbb N^n)
\]
the smallest closed $T$-invariant subspace containing $M$. A closed subspace $W \subset H$ is called a \textit{wandering subspace} for $T$ if
\[
W\perp T^k W\quad (k \in \mathbb N^n \setminus \lbrace 0 \rbrace).
\]
We say that $W$ is a \textit{generating wandering subspace} for $T$ if in addition $H = [W]$. 
If $W \subset M \subset H$ are closed subspaces of $H$ with $M \in {\rm Lat}(T)$, then obviously
$W$ is a wandering subspace for $T$ if and only if it is a wandering subspace for the restriction 
$T|M=(T_1|M,\ldots,T_n|M)$ of $T$ onto $M$. For each closed invariant subspace
$M \in {\rm Lat}(T)$, the space
\[
W_T(M) = M\ominus \sum^n_{i=1}T_iM
\]
is a wandering subspace for $T$. We call this space the \textit{wandering subspace associated with} 
$T|M$. If $W$ is a generating wandering subspace for $T$, then an elementary argument shows that necessarily
$W = W_T(H)$ is the wandering subspace associated with $T$. Indeed, if $w \in W$, 
then since each element in $H$ can be approximated by linear combinations of vectors of the
form $T^kh\ (k \in \mathbb N^n,h \in W)$, it follows that $w \in W_T(H)$. To prove the opposite inclusion, it suffices to observe that the orthogonal decomposition
\[
H = W\oplus \bigvee(T^kW;\ k \in \mathbb N^n \setminus \lbrace 0 \rbrace)
\]
holds for each wandering subspace $W$ that is generating for $T$. 
%Here $\bigvee_{i\in I}M_i$ denotes the closed linear span of the union of a family of subsets $M_i \subset H$.

\section{Homogeneous invariant subspaces}

Let $T \in B(H)^n$ be a commuting tuple of bounded linear operators on $H$. Throughout the whole paper 
we shall suppose that the underlying Hilbert space $H$ admits an orthogonal decomposition
$H = \bigcirc\!\!\!\!\!\!\!\!\perp^\infty_{k=0} H_k$ with closed subspaces $H_k\subset H$.
We call an element $h \in H$ homogeneous (of degree $k$) if $h \in H_k$ for some $k \in \mathbb N$. 
Let $\gamma = (\gamma_1,\ldots,\gamma_n)$ be an $n$-tuple of positive integers $\gamma_i> 0$. The
tuple $T \in B(H)^n$ is said to be of \textit{degree} $\gamma$ with respect to the above orthogonal 
decomposition of $H$, or simply $\gamma$-graded, if $T_i H_k\subset H_{k+\gamma_i}$ for $k\in \mathbb N$ and $i=1,\ldots,n$.
A closed subspace $M \subset H$ is called \textit{homogeneous} if
\[
M = \bigvee\limits^\infty_{k=0}M\cap H_k.
\]
An elementary exercise left to the reader shows that a closed linear subspace $M \subset H$ is homogeneous if
and only if one of the following equivalent conditions holds:\\
\begin{enumerate}[{\rm (i)}]
\item $P_M H_k \subset H_k \; \mbox{for all} \; k \in \mathbb N$,
\item $P_{H_k} M \subset M \; \mbox{for all} \; k \in \mathbb N$,
\end{enumerate}
where $P_M$ and $P_{H_k}$ are the orthogonal projections of $H$ onto $M$ and $H_k$.
For a homogeneous subspace $M \subset H$, the spaces $M_k = M\cap H_k$ are referred to as the \textit{homogeneous components} of $M$. For convenience, we define $M_k = \lbrace 0 \rbrace$ for each integer $k< 0$.
If $T \in B(H)^n$ is of degree $\gamma$ with respect to the decomposition
$H = \bigcirc\!\!\!\!\!\!\!\!\perp^\infty_{k=0} H_k$ and $M \in {\rm Lat}(T)$ is homogeneous, then the restriction
$T|M \in B(M)^n$ and the compression $P_{M^{\bot}} T| M^{\bot} \in B(M^{\bot})^n$ are of degree $\gamma$ 
with respect to the orthogonal decompositions
\[
M = \bigcirc\!\!\!\!\!\!\!\!\perp^\infty_{k=0} M \cap H_k \; \mbox{and} \; M^{\bot} = \bigcirc\!\!\!\!\!\!\!\!\perp^\infty_{k=0} M^{\bot} \cap H_k.
\]
%Lemma 1.1
\begin{lemma}\label{generating}
Let $T \in B(H)^n$ be a commuting tuple of degree $\gamma$. Then the wandering subspace $W = W_T(H)$ 
associated with $T$ is generating for $T$ and homogeneous with
\[
W_k = H_k\ominus \sum^n_{i=1}T_iH_{k-\gamma_i}\quad (k\geq 0).
\]
\end{lemma}

\begin{proof}
Let $h=\sum^\infty_{j=0}h_j\in W$ be given with $h_j\in H_j$ for all $j$. For $i=1,\ldots,n$ and
$x=\sum^\infty_{j=0}x_j\in H$ with $x_j \in H_j$ for all $j$, we obtain
\[
\langle h_k,T_ix\rangle=\langle h_k,T_ix_{k-\gamma_i}\rangle=\langle h,T_ix_{k-\gamma_i}\rangle=0
\]
for each $k$. Thus $W \subset H$ is a homogeneous subspace. If $h\in H_k\ominus\big(\sum^n_{i=1}T_iH_{k-\gamma_i}\big)$ 
and $x=\sum^\infty_{j=0}x_j \in H$ is a vector with $x_j \in H_j$ for all $j$, then
\[
\langle h,T_i x\rangle=\langle h,T_ix_{k-\gamma_i}\rangle=0\quad (i=1,\ldots,n).
\]
Thus $H_k\ominus\big(\sum^n_{i=1}T_iH_{k-\gamma_i}\big)\subset W_k$ for all $k$. The reverse inclusion is obvious.\\
To prove that $W$ is generating for $T$ we show by induction that
\[
H_j\subset [W]\mbox{ for all }j\geq 0.
\]
For $j=0$, clearly $H_0\subset W\subset [W]$. Let $j\geq 1$ be an integer such that the inclusion $H_k\subset [W]$ has been shown for $k=0,\ldots,j-1$. Then the observation that
\[
H_j=\left(H_j\ominus\sum^n_{i=1}T_iH_{j-\gamma_i}\right)\oplus\overline{\left(\sum^n_{i=1}T_iH_{j-\gamma_i}\right)}\subset W_j+[W]\subset [W]
\]
completes the proof.\qedhere
\end{proof}

Let $\tilde{H}  =\oplus^\infty_{k=0}H_k$ be the algebraic direct sum. If in the setting of 
Lemma \ref{generating} all the spaces $H_k$ $(k\geq 0)$ are finite dimensional, then the 
inductive argument used in its proof shows that
\[
\tilde{H}\subset {\rm span}\bigcup(T^kW; k\in \mathbb N^n),
\]
where the space on the right-hand side denotes the linear span of the spaces $T^kW$.\\
Let $M\subset H$ be a homogeneous closed subspace. Then the $T$-invariant subspace $[M] \in {\rm Lat}(T)$ 
generated by $M$ is homogeneous again. To see this, note that
\[
M = \bigvee_{k\in \mathbb N}M_k\subset \bigvee_{k\in \mathbb N}[M]\cap H_k
\]
and that the space on the right is a closed invariant subspace for $T$.

%Corollary 1.2
\begin{corollary}\label{bijection}
Let $T \in B(H)^n$ be a commuting tuple of degree $\gamma$. Then the mapping
\[
W \mapsto [W] = \bigvee_{k\in \mathbb N^n}T^kW
\]
defines a bijection between the homogeneous wandering subspaces for $T$ and the homogeneous 
closed $T$-invariant subspaces. The inverse of this mapping is given by
\[
M \mapsto W_T(M) = M\ominus \sum^n_{i=1}T_iM.
\]
\end{corollary}

\begin{proof}
By the remarks preceding Corollary \ref{bijection} the space $[W] \in {\rm Lat}(T)$ is 
homogeneous for each homogeneous wandering subspace $W$ of $T$. As seen in the introduction 
$W_T([W]) = W$ even for arbitrary wandering subspaces $W$ of $T$. If $M \in {\rm Lat}(T)$ is 
homogeneous, then Lemma \ref{generating} applied to the restriction $T|M$ shows that $W_T(M)$ 
is a homogeneous wandering subspace for $T$ with $[W_T(M)] = M$.\qedhere
\end{proof}

%With respect to the $\mathbb C[z]$-module structure induced by the tuple $T$ the Hilbert space
%$H$ becomes a graded module over the polynomial ring. To be more precise, 
Let us denote the space of all $\gamma$-\textit{homogeneous polynomials} of degree $k$ by
\[
\mathbb H_k = \mathbb H_k(\gamma) = \{ \sum_{\langle \alpha,\gamma \rangle = k} a_{\alpha} z^{\alpha}; \;
a_{\alpha} \in \mathbb C \mbox{ for all } \alpha \in \mathbb N^n \mbox{ with } \langle \alpha,\gamma \rangle = k \}.
\]
Here by definition $\langle \alpha,\gamma \rangle = \sum^n_{i=1}\alpha_i\gamma_i$.
Then the polynomial ring  $\mathbb C[z]$ decomposes into the direct sum
\[
\mathbb C[z] = \oplus_{k=0}^{\infty} \mathbb H_k.
\] 
The algebraic direct sum $\tilde{H} =  \oplus_{k=0}^{\infty} H_k$ equipped with the $\mathbb C[z]$-module structure
given by  
$
\mathbb C[z]\times \tilde{H}\rightarrow \tilde{H},(p,x)\mapsto p(T)x,
$
satisfies the conditions 
\[
\mathbb H_p H_q \subset H_{p+q} \quad (p,q \in \mathbb N).
\]
Hence $\tilde{H}$ is a \textit{graded module} over the \textit{graded ring} $\mathbb C[z]$.

Of particular interest for us will be the case that the $\mathbb C[z]$-module $\tilde{H}$ is 
finitely generated. It is well known (Theorem 4.12 in \cite{N}) that in this case
the linear subspaces $H_k \subset H$ are all finite dimensional.
An elementary argument left to the reader shows that
the $\mathbb C[z]$-module $\tilde{H}$ is finitely generated if and only if there is
a natural number $k_0 \in \mathbb N$ such that the spaces $H_k$ are finite dimensional for
$k = 0, \ldots, k_0$ and such that
\[
\bigvee_{\alpha \in \mathbb N^n}T^\alpha (\oplus^{k_0}_{k=0}) H_k = H.
\]

Recall that a commuting tuple $T\in B(H)^n$ is called $N$-\textit{cyclic}, for a given natural number 
$N \geq 1$, if there is a tuple $(x_1,\ldots,x_N) \in H^N$ with
\[
H = \bigvee\lbrace p(T)x_i;\ p \in \mathbb C[z]\mbox{ and } i=1,\ldots,N\rbrace
\]
and if $N$ is the minimal natural number with this property.

\begin{theorem}\label{minimal}
Let $T \in B(H)^n$ be a commuting tuple of degree $\gamma$. Then the following conditions are equivalent:
\begin{enumerate}[{\rm (i)}]
\item there is a finite set $\{x_1, \ldots, x_{\ell} \} \subset H$ with $H = [\{x_1, \ldots, x_{\ell} \} ]$,
\item $\dim W_T(H) < \infty$,
\item the $\mathbb C[z]$-module $\tilde{H}$ is finitely generated.
\end{enumerate}
In this case, $N = \dim\ W_T(H)$ is the minimal number of generators of the
$\mathbb C[z]$-module $\tilde{H} = \oplus^\infty_{k=0}H_k$, the tuple $T$ is $N$-cyclic and each basis of the vector space 
$W_T(H)$ generates $\tilde{H}$ as a $\mathbb C[z]$-module.
\end{theorem}

\begin{proof} The implication (iii) to (i) holds obviously. 
Indeed, if $\{x_1, \ldots, x_{\ell} \}$ is a finite generating set for the $\mathbb C[z]$-module $\tilde{H}$, then
$H = [\{x_1, \ldots, x_{\ell} \}]$.

Let us define $N = \dim W_T(H)$. Suppose that 
$(x_1,\ldots,x_\ell) \in H^\ell$ is an arbitrary finite tuple with $H = [\{x_1,\ldots,x_\ell\}]$.
Then for a vector $x \in H$, there are polynomials $p_{ij}\ (i=1,\ldots,\ell,j\in \mathbb N)$ such that
\[
x = \lim_{j\rightarrow \infty}\sum^\ell_{i=1}p_{ij}(T)x_i.
\]
Let $q: \, H \rightarrow H/\overline{\sum^n_{i=1}T_i H}$ be the quotient map. Then
\begin{align*}
q(x)&=\lim_{j\rightarrow \infty}\sum^\ell_{i=1} \big( p_{ij}(0)q(x_i)+q((p_{ij}(T)-p_{ij}(0))x_i) \big) \\
&=\lim_{j\rightarrow \infty}\sum^\ell_{i=1}p_{ij}(0)q(x_i)\in \bigvee\lbrace q(x_1),\ldots,q(x_\ell)\rbrace.
\end{align*}
Hence $N\leq \ell$. In particular, the number $N$ is finite.
To complete the proof let us suppose that $W_T(H)$ is
finite dimensional. Since by Lemma \ref{generating} the space $W_T(H)$ is the orthogonal direct 
sum of its homogeneous components $W_T(H) \cap H_k$, we obtain that $W_T(H) \subset \tilde{H}$.
Let $k \in \mathbb N$ and $x \in H_k$ be arbitrary. By Lemma \ref{generating} the vector $x$ is the
limit of a sequence of finite sums of the form
\[
\sum^r_{i=1} T^{\alpha_i} x_i, \quad x_i \in W_T(H) \cap H_{k_i}, \; \alpha_i \in \mathbb N^n.
\]
Since $T^{\alpha_i} x_i \in H_{\langle \alpha_i,\gamma \rangle + k_i}$, we find that $x = P_{H_k} x$ is
the limit of a sequence in the finite-dimensional and hence closed subspace
\[
\sum_{\langle \alpha,\gamma \rangle \leq k} T^{\alpha} W_T(H) \subset H.
\]
Thus $\tilde{H} = {\rm span}(\bigcup T^{\alpha} W_T(H); \; \alpha \in \mathbb N^n)$ and each
linear basis of the finite-dimensional space $W_T(H)$ generates $\tilde{H}$
as a $\mathbb  C[z]$-module. 
\qedhere
\end{proof}

Let us say that a commuting tuple $T \in B(H)^n$ is {\it finitely generated} and 
$\gamma$-{\it graded} if $T \in B(H)^n$ is of degree $\gamma$ with respect to the orthogonal 
decomposition $H = \bigcirc\!\!\!\!\!\!\!\!\perp^\infty_{k=0} H_k$ and satisfies one of the equivalent
conditions in Theorem \ref{minimal}. Note that if $T \in B(H)^n$ is finitely generated and
$\gamma$-graded, then so are each restriction $T|M$ of $T$ to a homogeneous closed $T$-invariant subspace $M$
and each compression $P_{M^{\bot}} T|M^{\bot}$ of $T$ to the orthogonal complement of $M$.
Indeed, the algebraic direct sum
\[
\tilde{M} = \oplus_{k=0}^{\infty} M_k  = \oplus_{k=0}^{\infty} M \cap H_k
\] 
is a finitely generated $\mathbb C[z]$-module as a submodule of the finitely generated
$\mathbb C[z]$-module $\tilde{H}$ (Proposition VI.1.4 in \cite{L}) and 
\[
(M^{\bot})^{\sim} = \oplus_{k=0}^{\infty} M^{\bot} \cap H_k \rightarrow \tilde{H}/\tilde{M}, \; x \mapsto x + \tilde{M}
\]
defines a $\mathbb C[z]$-module isomorphism onto the finitely generated $\mathbb C[z]$-module $\tilde{H}/\tilde{M}$.

\begin{corollary}\label{cyclic}
Let $T \in B(H)^n$ be a commuting tuple. Suppose that $T$ is finitely generated and $\gamma$-graded.
\begin{enumerate}[(a)]
\item Then the map
\[
W \mapsto [W] = \bigvee_{k\in \mathbb N^n}T^kW
\]
defines a bijection between the homogeneous wandering subspaces of $T$ and the homogeneous closed $T$-invariant subspaces with inverse acting as
		\[
		M\mapsto W_T(M) = M\ominus \sum^n_{i=1}T_iM.
		\]
		\item Each homogeneous wandering subspace $W$ for $T$ is finite dimensional. 
		If $M \in {\rm Lat}(T)$ is homogeneous, then $T|M$ is $N$-cyclic with
		\[
		N = \dim\ W_T(M) = \dim\ M\ominus \sum^n_{i=1}T_iM.
		\]
	\end{enumerate}
\end{corollary}

\begin{proof}
(a) As seen in Corollary \ref{bijection} part (a) holds for every commuting tuple $T \in B(H)^n$ of degree $\gamma$ without the extra condition that $H$ is finitely generated.\\
(b) Since $T|M$ is finitely generated and $\gamma$-graded again, part (b) follows from Theorem \ref{minimal}
applied to $T|M$. \qedhere
\end{proof}

Let $T \in B(H)^n$ be a commuting tuple that is finitely generated and $\gamma$-graded.
By Theorem \ref{minimal} there are finitely many natural numbers $0 \leq k_1 < k_2 < \ldots < k_r$ with
\[
\{ k \in \mathbb N ; \, W_T(H) \cap H_k \neq \{ 0 \} \} = \{ k_1, \ldots , k_r \}.
\]
By choosing a basis $(x_{ij})_{j=1, \ldots ,N_i}$ in each of the vector spaces
$W_T(H) \cap H_{k_i}$ $(i = 1, \ldots , r)$, one obtains a minimal set
\[
\{x_{11} , \ldots , x_{1N_1} , \ldots , x_{r1} , \ldots , x_{rN_r} \}
\]
of generators for the $\mathbb C[z]$-module $\tilde{H} = \oplus^\infty_{k=0}H_k$
consisting of homogeneous elements.

A well known result from commutative algebra says that, for a minimal set of homogeneous
generators of a graded module, the number of generators of any given degree $k$ is uniquely
determined. We give an elementary argument for this result in our setting.

For a finite set $F$, we denote by $\#(F)$ its number of elements.

\begin{lemma} \label{betti}
Let $T \in B(H)^n$ be a commuting tuple that is finitely generated and $\gamma$-graded.
Then we have:
\begin{enumerate}[(a)]
\item The map
\[
\rho: \, \tilde{H}/\sum^n_{i=1}T_i\tilde{H} \rightarrow  H/\overline{\sum^n_{i=1}T_iH},\ [x]\mapsto [x]
\]
defines an isomorphism of complex vector spaces.
\item If $N  = \dim W_T(H)$ and and if the elements $x_i \in H_{k_i}$ $(i = 1, \ldots ,N)$ form a 
system of generators for the $\mathbb C[z]$-module $\tilde{H}$, then for each $k \geq 0$, 
\[ 
\#(\{i = 1, \ldots , N; \; k_i =k \}) = \dim W_T(H) \cap H_k.
\]
\end{enumerate}
\end{lemma}

\begin{proof}
Define $N = \dim\ W_T(H)$ and fix a generating set $\{ x_1,\ldots,x_N \}$ 
for the $\mathbb C[z]$-module $\tilde{H}$.
Then for each element $[x]\in \tilde{H}/\sum^n_{i=1}T_i\tilde{H}$, there are polynomials
$p_1, \ldots ,p_N \in \mathbb C[z]$ with
\[
[x] = \sum^N_{i=1}[p_i(T)x_i] = \sum^N_{i=1}p_i(0)[x_i]\in {\rm span}\lbrace [x_1],\ldots,[x_N]\rbrace.
\]
Hence $\dim \tilde{H}/\sum^n_{i=1}T_i\tilde{H}\leq N$. Let $P \in B(H)$ be the orthogonal 
projection onto $W_T(H)$. Then for $x \in H$,
\[
x-P x \in \overline{\sum^n_{i=1}T_i H}\mbox{ with }P x\in W_T(H)\subset \tilde{H}.
\]
Hence the linear map
\[
\rho: \, \tilde{H}/\sum^n_{i=1}T_i\tilde{H} \rightarrow  H/\overline{\sum^n_{i=1}T_i H} \cong W_T(H),\ [x]\mapsto [x]
\]
is onto. But then also $N = \dim\ W_T(H) \leq \dim\ \tilde{H}/\sum^n_{i=1}T_i\tilde{H}$. Thus $\rho$ is a vector space ismorphism
and the equivalence classes $[x_1], \ldots, [x_N]$ form a basis of the vector space
$\tilde{H}/\sum^n_{i=1}T_i\tilde{H}$. Suppose in addition that $x_i \in H_{k_i}$ for 
$i = 1, \ldots ,N$. Since $P H_k \subset H_k$ for all $k$, it follows that, for each $k \geq 0$, the elements in
\[
\lbrace P x_i;\ i = 1,\ldots,N\mbox{ with }k_i = k \rbrace
\]
form a basis of the vector space $W_T(H) \cap H_k$.
\end{proof}

\begin{corollary}\label{orthogonal}
Let $T\in B(H)^n$ be a commuting tuple that is finitely generated and $\gamma$-graded. Then we have:
\begin{enumerate}[(a)]
	\item $ \tilde{H} = W_T(H) \oplus \sum^n_{i=1} T_i \tilde{H}$,
	\item if $x_1,\ldots,x_N \in \tilde{H}$ generate $\tilde{H}$ as a $\mathbb C[z]$-module and if $x \in H$ is arbitrary, 
	then $x \in (\sum^n_{i=1}T_iH)^\perp$ if and only if
	\[
	\langle x,T^\alpha x_i\rangle=0 \mbox{ for } i=1,\ldots,N \mbox{ and }\alpha \in \mathbb N^n \setminus \lbrace 0 \rbrace.
	\]
\end{enumerate} 
\end{corollary}

\begin{proof}
Since $W_T(H)$ is finite dimensional and homogeneous, it follows that  $W_T(H) \subset \tilde{H}$. 
To complete the proof of part (a) note that with the notation from the proof of Lemma \ref{generating}
\[
H_j = \left(H_j \ominus \sum^n_{i=1}T_i H_{j-\gamma_i} \right) \oplus \overline{\left(\sum^n_{i=1}T_iH_{j-\gamma_i}\right)}
    = W_j \oplus \sum^n_{i=1}T_i H_{j-\gamma_i}.
\]		
To prove the non-trivial implication in part (b), fix a vector $x \in H$ with $\langle x,T^\alpha x_i\rangle=0$ for
$i=1,\ldots,N$ and $\alpha \in \mathbb N^n\setminus \lbrace 0 \rbrace$. By hypothesis each element in $\sum^n_{i=1}T_i H$ is the limit of a sequence of elements of the form
\[
\sum^n_{i=1}\sum^N_{j=1}T_iq_{ij}(T)x_j
\]
with $q_{ij}\in \mathbb C[z]$. It follows that $x$ is orthogonal to the space $\sum^n_{i=1}T_i H$.
\end{proof}

Let $T \in B(H)^n$ be as in Corollary \ref{orthogonal}. Suppose that the $\mathbb C[z]$-module 
$\tilde{H} = \oplus^\infty_{k=0} H_k$ is generated by non-zero homogeneous elements
\[
x^1_1, \ldots  ,x^1_{N_1}\in H_{k_1}, x^2_1, \ldots , x^2_{N_2}\in H_{k_2}, \ldots , x^r_1,\ldots,x^r_{N_r} \in H_{k_r},
\]
where $0 \leq k_1 < k_2 < \ldots < k_r$ and $N_1,\ldots, N_r \geq 1$ are given integers. 
To simplify the notations define $N = N_1 + \ldots + N_r$ and
\[
(x_1,\ldots,x_N) = (x^1_1,\ldots,x^1_{N_1}, \ldots , x^r_1,\ldots,x^r_{N_r}).
\]

\begin{lemma}\label{sufficient}
Let $(x_1,\ldots,x_N) = (x^1_1,\ldots,x^1_{N_1}, \ldots , x^r_1,\ldots,x^r_{N_r})$
be non-zero generators of the $\mathbb C[z]$-module $\tilde{H}=\oplus^\infty_{k=0} H_k$ with
$x^i_j\in H_{k_i}$ for $i = 1, \ldots,r,j = 1,\ldots,N_i$, where  $0 \leq k_1 < k_2 < \ldots < k_r$. 
Suppose that
\[
\langle x^q_i,T^\alpha x^p_j\rangle = 0
\]
for all integers $p,q,i,j$ with $1 \leq p \leq q \leq r, 1 \leq i \leq N_q, 1 \leq j \leq N_p$ 
and every multi-index $\alpha \in \mathbb N^n$ with $\langle \alpha,\gamma\rangle = k_q - k_p$.
Then the vectors $x_1,\ldots,x_N$ form an orthogonal basis of $W_T(H)$.
\end{lemma}

\begin{proof}
The vectors $x_1,\ldots,x_N \in \tilde{H}$ are pairwise othogonal and satisfy the relations
\[
\langle x_i,T^\alpha x_j\rangle=0
\]
for all $i,j =1,\ldots,N$ and $\alpha \in \mathbb N^n\setminus \lbrace 0 \rbrace$. 
Thus by Corollary \ref{orthogonal} the vectors $x_1,\ldots,x_N$ form a linearly independent set in $W_T(H)$.
Since by hypothesis these vectors generate the $\mathbb C[z]$-module $\tilde{H}$, Theorem \ref{minimal} implies 
that they form a basis of $W_T(M)$.
\end{proof}

Exactly as before, let us suppose that the vectors
\[
(x_1,\ldots,x_N) = (x^1_1,\ldots,x^1_{N_1},\ldots,x^r_1,\ldots,x^r_{N_r})
\]
form a set of non-zero homogeneous generators
\[
x^i_j \in H_{k_i}\qquad (i=1,\ldots,r,j=1,\ldots,N_i)
\]
for the $\mathbb C[z]$-module $\tilde{H}$ with $0 \leq k_1 < \ldots < k_r$. The above observations can be used to formulate an algorithm which produces a minimal set of homogeneous generators for $\tilde{H}$.
Suppose that there is an integer $1 \leq s < r$ such that the orthogonality relations
\[
\langle x^q_i,T^\alpha x^p_j\rangle = 0
\]
hold for $1 \leq p \leq q \leq s-1, 1 \leq i \leq N_q, 1 \leq j \leq N_p$ and all multi-indices 
$\alpha \in \mathbb N^n$ with $\langle \alpha,\gamma\rangle=k_q-k_p$. Let $y_1,\ldots,y_M$ be an enumeration
of the set
\[
\lbrace T^\alpha x^p_j;\ 1 \leq p < s,1 \leq j \leq N_p,
\alpha \in \mathbb N^n \mbox{ with } \langle \alpha,\gamma\rangle = k_s-k_p\rbrace.
\]
If $x^s_1,\ldots,x^s_{N_s}\in \mbox{ span}\lbrace y_1,\ldots,y_M\rbrace$, then we remove 
these vectors from the original set $(x_1,\ldots,x_N)$ of generators. Otherwise we apply
the \textit{Gram-Schmidt orthogonalization procedure} to the vectors 
\[
y_1,\ldots,y_M,x^s_1,\ldots,x^s_{N_s}
\]
to obtain pairwise orthogonal non-zero vectors
\[
y^s_1,\ldots,y^s_{n_s} \in \lbrace y_1,\ldots,y_M \rbrace^\perp
\]
with $n_s \leq N_s$ such that
\[
\mbox{span}\lbrace y_1,\ldots,y_M,y^s_1,\ldots,y^s_{n_s}\rbrace=\mbox{span}\lbrace y_1,\ldots,y_M,x^s_1,\ldots,x^s_{N_s}\rbrace.
\]
Now replace $x^s_1,\ldots,x^s_{N_s}$ by the vectors $y^s_1,\ldots,y^s_{n_s}$ obtained in this way.  It is elementary
to check that the two sets
\[
\{x^p_i; \; p = 1, \ldots, s-1, i = 1, \ldots, N_p \} \cup \{ y^s_1,\ldots,y^s_{n_s} \}
\]
and
\[
\{x^p_i; \; p = 1, \ldots, s, i = 1, \ldots, N_p \}
\]
generate the same $\mathbb C[z]$-submodule $M_p$ of $H$.\\
By applying the above procedure finitely many times one obtains a set 
\[
y^1_1,\ldots,y^1_{n_1},\ldots,y^{\rho}_1,\ldots,y^{\rho}_{n_{\rho}}
\]
of non-zero homogeneous generators of the $\mathbb C[z]$-module $\tilde{H}$ which satisfies the hypotheses
of Lemma \ref{sufficient}. As an orthogonal basis of the wandering subspace $W_T(H)$ this set is a minimal
generating set for the $\mathbb C[z]$-module $\tilde{H}$.\\

The above algorithm can also be used to construct a minimal generating set for $\tilde{H}$ that is
contained in the original set $\{x_1, \ldots, x_N\}$ of homogeneous generators. Indeed, note that 
with the above notations, the vectors $y^s_{\mu}$ $(\mu = 1, \ldots, n_s)$ are of the form
\[
x^s_{i_{\mu}} - \sum^M_{i=1} \alpha_i y_i - \sum^{\mu - 1}_{\nu = 1} \beta_i y^s_{\nu}
\]
with suitable $1 \leq i_1 < \ldots < i_{n_s} \leq N_s$. Hence the submodule $M_p$ is also generated by
the set
\[
\{x^p_i; \; p = 1, \ldots, s-1, i = 1, \ldots, N_p \} \cup \{ x^s_{i_1},\ldots,x^s_{i_{n_s}} \}.
\]
Therefore in the above inductive procedure the vectors $x^s_1,\ldots,x^s_{N_s}$ can also be replaced
by the vectors $x^s_{i_1}, \ldots, x^s_{i_{n_s}}$.\\

Let $H^2(\mathbb D^n)$ be the Hardy space on the unit polydisc $\mathbb D^n \subset \mathbb C^n$. Then
the Hilbert space $H^2(\mathbb D^n)$ admits the orthogonal decomposition
\[
H^2(\mathbb D^n) = \bigcirc\!\!\!\!\!\!\!\!\perp ^\infty_{k=0} \mathbb H_k,
\]
where $\mathbb H_k$ is the space of all homogeneous polynomials of degree $k$. The coordinate functions
$z_i$ are multipliers of $H^2(\mathbb D^n)$ and the induced multiplication operators
$M_{z_i}: H^2(\mathbb D^n) \rightarrow H^2(\mathbb D^n)$ satisfy
\[
M_{z_i} \mathbb H_k \subset \mathbb H_{k+1} \quad (k \in \mathbb N, i = 1, \ldots, n).
\]
Hence all the preceding results apply to $M_z = (M_{z_1}, \ldots, M_{z_n})
\in B(H^2(\mathbb D^n))^n$. If $I \subset \mathbb C[z]$ is a homogeneous ideal, then its closure
$M = \overline{I} \subset H^2(\mathbb D^n)$ is a homogeneous $M_z$-invariant subspace. The ideal
$I$ can be recovered from $M$ as the algebraic direct sum
\[
I = \tilde{M} = \oplus^{\infty}_{k=0} M \cap \mathbb H_k.
\]
Hence the above algorithm applies and transforms each set $(p_1, \ldots, p_N)$ of 
homogeneous generators of $I$ into a minimal
set of homogeneous generators for $I$. The Hardy space $H^2(\mathbb D^n)$ can be replaced by any
other Hilbert space completion of $\mathbb C[z]$ yielding the above graded structure such as
the standard analytic functional Hilbert spaces $H_m(\mathbb B) \subset \mathcal O(\mathbb B)$
on the unit ball $\mathbb B \subset \mathbb C^n$ considered in \cite{E}. 

\section{Free resolutions of graded row contractions}

Let $T  \in B(H)^n$ as before be a finitely generated $\gamma$-graded commuting tuple. 
In this section we suppose in
addition that $T$ is a \textit{row contraction}, that is, a commuting tuple $T  \in B(H)^n$ with
\[
\sum^n_{i=1} T_i T^*_i \leq 1_H.
\]
We extend ideas of Arveson \cite{A} to the $\gamma$-graded case to construct weak resolutions of $T$.

For an arbitray Hilbert space $\mathcal E$, we denote by $H^2_n(\mathcal E)$ the $\mathcal E$-valued
\textit{Drury-Arveson space}, that is, the analytic functional Hilbert space $H^2_n(\mathcal E) \subset \mathcal O(\mathbb B,\mathcal E)$
defined by the reproducing kernel
\[
K_{\mathcal E} : \mathbb B \times \mathbb B \rightarrow B(\mathcal E), \; (z,w) \mapsto \frac{1_{\mathcal E}}{1 - \langle z,w \rangle}.
\]
We write $M_z = M_z^{\mathcal E} \in B(H^2_n(\mathcal E))^n$ for the commuting tuple consisting of the multiplication operators
with the coordinate functions.

Since $T$ is $\gamma$-graded, it follows that $T$ is a \textit{pure row contraction}, that is,
\[
{\rm SOT}- \lim_{j \rightarrow \infty} \sigma^j_T(1_H) = 
{\rm SOT}- \lim_{j \rightarrow \infty} \sum_{|\alpha| = j} \frac{|\alpha|!}{\alpha!} T^{\alpha} T^{*\alpha} = 0,
\]
where $\sigma_T: B(H) \rightarrow B(H)$ acts as $\sigma_T(X) = \sum^n_{i=1} T_i X T^*_i$.
Indeed, since $\| \sigma^k_T(1_H) \| \leq \| \sigma_T \|^k \leq 1$ for all $k$, 
it suffices to check that
\[
\sum_{|\alpha| = j} \frac{|\alpha|!}{\alpha!} T^{\alpha} T^{*\alpha} f 
\stackrel{(j \rightarrow \infty)}{\longrightarrow} 0
\]
for $f \in H_k$ $(k \geq 0)$. But this is obvious, since $T^*_i H_k \subset H_{k-\gamma_i}$ for $k \in \mathbb N$
and $i = 1, \ldots, n$, where $H_{k-\gamma_i} = \{0\}$ for $k-\gamma_i < 0$.

By a well known invariant subspace result (see e.g. Theorem 4.1 in \cite{BEKS}) there are a Hilbert space $\mathcal E^0$ and
a surjective partial isometry $\Pi^0 \in B(H^2_n(\mathcal E^0),H)$ with $\Pi^0 M_{z_i} = T_i \Pi^0$ for
$i = 1, \ldots, n$. We define
\[
\mathcal D^0 = ({\rm Ker} \, \Pi^0)^{\bot} \cap \mathcal E^0,
\]
where we regard $\mathcal E^0 \subset H^2_n(\mathcal E^0)$ as the closed subspace consisting of the
constant functions. By Theorem 5.2 in \cite{BEKS} the map $\Pi^0$ induces a unitary operator
\[
\Pi^0: \; \mathcal D^0 \rightarrow W_T(H), x \mapsto \Pi^0 x.
\]
By Theorem \ref{minimal} there is an integer $K \geq 0$ such that
\[
W_T(H) = \oplus^K_{k=0} W_T(H) \cap H_k.
\]
Then $\mathcal D^0$ has the orthogonal decomposition
\[
\mathcal D^0 = \oplus^K_{k=0} \mathcal D^0_k
\]
with $\mathcal D^0_k = \{ x \in \mathcal D^0; \Pi^0 x \in H_k \}$. The space $H^0 = H^2_n(\mathcal D^0)$
admits the decomposition
\[
H^0 = \bigcirc\!\!\!\!\!\!\!\!\perp^\infty_{k=0} H^0_k \; \mbox{ with } \; 
H^0_k = \oplus_{i+j=k} \mathbb H_i \otimes \mathcal D^0_j,
\]
where $\mathbb H_i$ denotes as before the space of all $\gamma$-homogeneous polynomials of degree $i$. Because of 
\[
z_{\nu} H^0_k \subset \sum_{i+j=k} \mathbb H_{i+\gamma_{\nu}} \otimes \mathcal D^0_j \subset H^0_{k+\gamma_{\nu}}
\]
for $\nu = 1, \ldots,n$ and $k \geq 0$ the tuple $M_z \in B(H^0)^n$ is of degree $\gamma$ with respect to the
decomposition $H^0 = \bigcirc\!\!\!\!\!\!\!\!\perp^\infty_{k=0} H^0_k$. The algebraic direct sum
\[
\tilde{H}^0 = \oplus_{k \geq 0} H^0_k = \mathbb C[z] \otimes \mathcal D^0
\]
is a finitely generated $\mathbb C[z]$-module again. Since
\[
\Pi^0 ( \sum_{i+j=k} p_i \otimes x_j) = \sum_{i+j=k} p_i(T) (\Pi^0 x_j) \in H_k
\]
for $k \in \mathbb N, p_i \in \mathbb H_i$ and $x_j \in \mathcal D^0_j$, the map
$\Pi^0: H^0 = H^2_n(\mathcal D^0) \rightarrow H$ is continuous linear with dense range, intertwines the
tuples $M_z \in B(H^0)^n$ and $T \in B(H)^n$ componentwise and respects the gradings on both sides in 
the sense that
\[
\Pi^0 (H^0_k) \subset H_k \quad \mbox{ for } k \geq 0.
\]
The induced $\mathbb C[z]$-module homomorphism $ \Pi^0 : \tilde{H}^0 \rightarrow \tilde{H}$ is \textit{minimal}
in the usual algebraic sense, that is, satisfies the inclusion
\[
{\rm Ker}(\Pi^0 : \tilde{H}^0 \rightarrow \tilde{H}) \subset \sum^n_{i=1} z_i \tilde{H}^0.
\]
Indeed, by the definition of $\mathcal D^0$ and part (a) of Corollary \ref{orthogonal}, we obtain
\[
{\rm Ker}(\Pi^0 : \tilde{H}^0 \rightarrow \tilde{H}) \subset \tilde{H}^0 \ominus \mathcal D^0 = \sum^n_{i=1} z_i \tilde{H}^0.
\]

The kernel $M^0 = {\rm Ker}(H^0 \stackrel{\Pi^0}{\longrightarrow} H)$ is a homogeneous closed
invariant subspace for $M_z \in B(H^0)^n$ and $T^0 = M_z|M^0 \in B(M^0)^n$ is a commuting row
contraction of degree $\gamma$ on $M^0$ such that $\tilde{M}^0 \subset \tilde{H}^0$ is a finitely 
generated $\mathbb C[z]$-submodule. Since $T^0$ satisfies the same hypotheses as $T \in B(H)^n$, we can
repeat the process. In this way we obtain a sequence
\[
\ldots \stackrel{\Pi^2}{\longrightarrow} H^2_n(\mathcal D^1) \stackrel{\Pi^1}{\longrightarrow} H^2_n(\mathcal D^0)
\stackrel{\Pi^0}{\longrightarrow} H \rightarrow 0
\]
intertwining componentwise the tuples $M_z \in B(H^2_n(\mathcal D^i))^n $ $(i \geq 0)$ and $T\in B(H)^n$
such that
\[
\overline{{\rm Im} \, \Pi^0} = H \; \mbox{ and } \overline{{\rm Im} \, \Pi^i} = {\rm Ker}(\Pi^{i-1}) \; (i \geq 1)
\]
and
\[
\Pi^0(\mathcal D^0) = W_T(H) \; \mbox{ and } \Pi^i(\mathcal D^i) = W_{M_z}({\rm Ker} \, \Pi^{i-1}) \; (i \geq 1).
\]

We write $H^i = H^2_n(\mathcal D^i)$ $(i \geq 0)$. Then the tuples $M_z \in B(H^i)^n$ are finitely
generated $\gamma$-graded and the maps $\Pi^i: H^i \rightarrow H^{i-1}$ $(i \geq 1)$ respect the gradings. The
induced complex
\[
\ldots \stackrel{\Pi^2}{\longrightarrow} \tilde{H}^1 = \mathbb C[z] \otimes \mathcal D^1 \stackrel{\Pi^1}{\longrightarrow} 
\tilde{H}^0 = \mathbb C[z] \otimes \mathcal D^0 \stackrel{\Pi^0}{\longrightarrow} \tilde{H} \rightarrow 0
\]
consists of minimal $\mathbb C[z]$-module homomorphisms. The following lemma shows that this complex defines
a \textit{minimal graded free resolution} of the finitely generated graded $\mathbb C[z]$-module $\tilde{H}$.

\begin{lemma} Let $S \in B(L)^n$, $T \in B(H)^n$ and $R \in B(M)^n$ be finitely generated $\gamma$-graded commuting
tuples on Hilbert spaces $L, H$ and $M$. Suppose that
\[
L \stackrel{\alpha}{\longrightarrow} H \stackrel{\beta}{\longrightarrow} M
\]
is a complex of continuous linear operators that intertwine the tuples $S \in B(L)^n$, $T \in B(H)^n$ and $R \in B(M)^n$ 
componentwise, respect the gradings
and satisfy $\overline{{\rm Im} \, \alpha} = {\rm Ker} \, \beta$. Then the induced complex
\[
\tilde{L} \stackrel{\alpha}{\longrightarrow} \tilde{H} \stackrel{\beta}{\longrightarrow} \tilde{M}
\]
of finitely generated $\mathbb C[z]$-modules is exact.
\end{lemma}

\begin{proof}
Let $x \in H_k$ be an element with $\beta x = 0$. Then there is a sequence $(x_j)_{j \geq 0}$ in $L$
such that $(\alpha x_j)_{j \geq 0}$ converges to $x$. Let $P_{H_k}$ and $P_{L_k}$ be the orthogonal
projections of $H$ onto $H_k$ and $L$ onto $L_k$. Since $\alpha$ respects the
gradings, it follows that
\[
\alpha(P_{L_k}x_j) = P_{H_k}(\alpha x_j)  \stackrel{j}{\rightarrow} x.
\]
Since $H_k$ is finite dimensional, the subspace $\alpha(L_k) \subset H_k$ is closed and hence
$x \in \alpha(L_k)$. Thus the assertion follows.
\end{proof}

Let $T \in B(H)^n$ be a commuting tuple on a Hilbert space $H$ and let $(H^i,T^i)_{i \geq 0}$ be a finite
or infinite family of Hilbert spaces $H^i$ and commuting tuples $T^i \in B(H^i)^n$ such that there is a complex
\[
\ldots \stackrel{\Pi^2}{\longrightarrow} H^1 \stackrel{\Pi^1}{\longrightarrow} H^0 \stackrel{\Pi^0}{\longrightarrow} H \rightarrow 0
\]
of bounded linear operators $\Pi^i$ that intertwine the tuples  $T^i \in B(H^i)^n$ and $T \in B(H)^n$ componentwise.
Following Douglas and Misra \cite{DM} and Arveson \cite{A} we call 
$(H^{\bullet},\Pi^{\bullet},T^{\bullet})$ a \textit{weak resolution} of $T \in B(H)^n$ if
\[
 \overline{\Pi^0 H^0} = H \; \mbox{ and } \; \overline{\Pi^i H^i} = {\rm Ker}\, \Pi^{i-1} \quad (i \geq 1)
\]
and if $\Pi^0 : H^0 \rightarrow H$ and $\Pi^i : \, H^i \rightarrow H^{i-1}$ $(i \geq 1)$ induce unitary operators
\[
\Pi^0 : \; W_{T^0}(H^0) \rightarrow W_T(H) \; \mbox{ and } \; 
\Pi^i: \; W_{T^i}(H^i) \rightarrow W_{T^{i-1}}({\rm Ker}\, \Pi^{i-1}) \quad (i \geq 1).
\]
We call two weak resolutions $(H^{\bullet},\Pi^{\bullet},T^{\bullet})$ and 
$(K^{\bullet},\Gamma^{\bullet},S^{\bullet})$ \textit{isomorphic}
if there are unitary operators $U^i: \ H^i \rightarrow K^i$ that intertwine the tuples $T^i \in B(H^i)^n$,
$S^i \in B(K^i)^n$ componentwise and satisfy the relations $\Pi^0 = \Gamma^0 U^0$,
$U^{i-1} \Pi^i = \Gamma^i U^i$ $(i \geq 1)$.\\

Since $W_{M_z}(H^2_n(\mathcal D^i)) = \mathcal D^i$, the family $(H^2_n(\mathcal D^i),M_z^{\mathcal D^i})$ constructed above 
for the finitely generated $\gamma$-graded commuting row contraction $T \in B(H)^n$ defines a weak resolution
\[
\ldots \stackrel{\Pi^2}{\longrightarrow} H^2_n(\mathcal D^1) \stackrel{\Pi^1}{\longrightarrow} H^2_n(\mathcal D^0)
\stackrel{\Pi^0}{\longrightarrow} H \rightarrow 0
\]
of $T$. By a weak resolution of $T$ of the form $(H^2_n(\mathcal D^{\bullet}),\Pi^{\bullet})$ we shall always mean 
a weak resolution $(H^2_n(\mathcal D^{\bullet}),\Pi^{\bullet},M_z^{\mathcal D^{\bullet}})$.

\begin{theorem}\label{resolutions}
Let $T \in B(H)^n$ be a finitely generated $\gamma$-graded commuting row contraction. Then $T$ admits a finite weak
resolution of the form
\[
0 \rightarrow H^2_n(\mathcal D^n) \stackrel{\Pi^n}{\longrightarrow} \ldots \stackrel{\Pi^2}{\longrightarrow} 
H^2_n(\mathcal D^1) \stackrel{\Pi^1}{\longrightarrow} H^2_n(\mathcal D^0)
\stackrel{\Pi^0}{\longrightarrow} H \rightarrow 0.
\]
Any two weak resolutions of the form $(H^2_n(\mathcal D^{\bullet}),\Pi^{\bullet})$ and 
$(H^2_n(\mathcal E^{\bullet}),\Gamma^{\bullet})$ are isomorphic.
\end{theorem}

\begin{proof} 
To prove the existence part, we show that the weak resolution
$(H^2_n(\mathcal D^{\bullet}),\Pi^{\bullet})$ of $T$ constructed above is finite. Let 
$(\mathbb C[z] \otimes \mathcal D^{\bullet},\Pi^{\bullet})$ be the induced minimal graded free resolution of
$\tilde{H}$. By Hilbert's Syzygy Theorem (Theorem VII.43 in \cite{ZS}) the space
${\rm Ker} \, \Pi^{n-1}$ is a finitely generated free $\mathbb C[z]$-module. By the uniqueness of minimal free graded
resolutions the resolution $(\mathbb C[z] \otimes \mathcal D^{\bullet},\Pi^{\bullet})$ is isomorphic to the
finite resolution
\[
0 \rightarrow {\rm Ker} \, \Pi^{n-1} \rightarrow  \mathbb C[z] \otimes \mathcal D^{n-1} \stackrel{\Pi^{n-1}}{\longrightarrow} \ldots 
\stackrel{\Pi^2}{\longrightarrow} \mathbb C[z] \otimes \mathcal D^1 \stackrel{\Pi^1}{\longrightarrow} \mathbb C[z] \otimes \mathcal D^0
\stackrel{\Pi^0}{\longrightarrow} \tilde{H} \rightarrow 0.
\]
Hence $\mathcal D^i = 0$ for $i \geq n+1$.

Let $(H^2_n(\mathcal D^{\bullet}),\Pi^{\bullet})$ and $(H^2_n(\mathcal E^{\bullet}),\Gamma^{\bullet})$ be weak resolutions
of $T$. Then $\Pi^0$ and $\Gamma^0$ induce unitary operators $\Pi^0: \mathcal D^0 = W_{M_z}(H^2_n(\mathcal D^0)) \rightarrow W_T(H)$
and $\Gamma^0: \mathcal E^0 = W_{M_z}(H^2_n(\mathcal E^0)) \rightarrow W_T(H)$. Hence there is a unitary operator
$U^0: \mathcal D^0 \rightarrow \mathcal E^0$ such that $\Pi^0 = \Gamma^0 U^0$. Since the polynomials are dense in $H^2_n$,
it follows that $\Pi^0$ acts as the composition
\[
H^2_n(\mathcal D^0) \stackrel{1 \otimes U^0}{\longrightarrow} H^2_n(\mathcal E^0) \stackrel{\Gamma^0}{\longrightarrow} H.
\]
Similarly, one obtains a unitary operator $U^1: \mathcal D^1 \rightarrow \mathcal E^1$ such that 
$\Gamma^1 (1_{H^2_n} \otimes U^1) = (1_{H^2_n} \otimes U^0) \Pi^1$. Inductively one obtains unitary
operators $1_{H^2_n} \otimes U^i: H^2_n(\mathcal D^i) \rightarrow H^2_n(\mathcal E^i)$ that establish an isomorphism
between $(H^2_n(\mathcal D^{\bullet}),\Pi^{\bullet})$ and $(H^2_n(\mathcal E^{\bullet}),\Gamma^{\bullet})$.
\qedhere
\end{proof}

%Let $T \in B(H)^n$ be a finitely generated $\gamma$-graded commuting row contraction. Fix a weak resolution
%\[
%0 \rightarrow H^2_n(\mathcal D^n) \stackrel{\Pi^n}{\longrightarrow} \ldots \stackrel{\Pi^2}{\longrightarrow} 
%H^2_n(\mathcal D^1) \stackrel{\Pi^1}{\longrightarrow} H^2_n(\mathcal D^0)
%\stackrel{\Pi^0}{\longrightarrow} H \rightarrow 0
%\]
%of $T$ as in Theorem \ref{resolutions}. 
Let us fix a weak resolution $(H^2_n(\mathcal D^{\bullet}),\Pi^{\bullet})$ of $T$ and the induced minimal graded
free resolution $(\mathbb C[z] \otimes \mathcal D^{\bullet},\Pi^{\bullet})$ of $\tilde{H}$.
Since the operators $\Pi^i: H^2_n(\mathcal D^i) \rightarrow H^2_n(\mathcal D^{i-1})$
intertwine the multiplication tuples $M_z$ on $H^2_n(\mathcal D^{i})$ and $H^2_n(\mathcal D^{i-1})$, there
are operator-valued analytic functions $\varphi^i: \mathbb B \rightarrow B(\mathcal D^i,\mathcal D^{i-1})$ such that
$\Pi^i$ acts as the multiplication operator $\Pi^i f = \varphi^i f$ (\cite{B}) for $i \geq 1$. An elementary argument
using the finite dimensionality of the spaces $\mathcal D^i$ and the inclusions $\varphi^i (\mathbb C[z] \otimes \mathcal D^i)
\subset \mathbb C[z] \otimes \mathcal D^{i-1}$ shows that the functions $\varphi^i$ are in fact operator-valued polynomials. For
convenience, let us set $\varphi^0 = 0$.

For a module $M$ and a commuting tuple $a = (a_1, \ldots, a_n)$ of module homomorphisms $a_i: M \rightarrow M$,
we denote by
\[
K_{\bullet}(a,M): 0 \rightarrow \Lambda^n M \stackrel{\delta_n^a}{\longrightarrow} \Lambda^{n-1} M
\stackrel{\delta_{n-1}^a}{\longrightarrow} \ldots \stackrel{\delta_1^a}{\longrightarrow} \Lambda^0 M \rightarrow 0
\]
the Koszul complex of $a$ and by $H_p(a,M) =  {\rm Ker} \, \delta_p^a/{\rm Im} \, \delta_{p+1}^a$ 
its homology groups (cf. Chapter 2 in \cite{EP}). 
%Our next aim is to compare the
%homology groups $H_p(\lambda - T,H) = {\rm Ker} \ \delta_p^{\lambda - T}/{\rm Im} \ \delta_{p+1}^{\lambda - T}$
%with the homology groups $H_p(\mathcal D^{\bullet},\varphi^{\bullet}(\lambda)) = 
%{\rm Ker} \ \varphi^p(\lambda)/{\rm Im}\ \varphi^{p+1}(\lambda)$. To this end, 

Let $K = (K_{p,q},\partial',\partial'')$ be the double complex with 
$K_{p,q} = K_p(z - \lambda,\mathbb C[z] \otimes \mathcal D^q)$, $q$-th column $(K_{\bullet,q},\partial'_{\bullet})$ 
equal to $(-1)^q$ times the Koszul complex $K_{\bullet}(z - \lambda,\mathbb C[z] \otimes \mathcal D^q)$ 
and $p$-th row $(K_{p,\bullet},\partial''_{\bullet})$ given by the
$\binom{n}{p}$-fold direct sum of the complex $(\mathbb C[z] \otimes \mathcal D^{\bullet},\varphi^{\bullet})$.
Since the canonically augmented complexes
\[
(K_{\bullet,q},\partial'_{\bullet}) \rightarrow \mathcal D^q \rightarrow 0 \; \mbox{ and } \;
(K_{p,\bullet},\partial''_{\bullet}) \rightarrow \Lambda^p\tilde{H} \rightarrow 0
\]
are exact and $K$ has bounded diagonals, standard double complex arguments (Lemma A2.6 in \cite{EP}) 
show that there are induced vector space isomorphisms
\[
H_p(z -\lambda,\tilde{H}) \cong H'_p H''_0(K) \cong H''_p H'_0(K)
\cong H_p(\mathcal D^{\bullet},\varphi^{\bullet}(\lambda)) \quad (p \geq 0).
\]
By Theorem 2.3 from \cite{Comp} there are
vector space isomorphisms $H_p(z,\tilde{H}) \cong \overline{H}_p(T,H)$ where by definition
\[
\overline{H}_p(T,H) = {\rm Ker} \, \delta_p^{T}/\overline{{\rm Im} \, \delta_{p+1}^{T}}.
\]

Recall that a commuting tuple $S = (S_1, \ldots, S_n) \in B(H)^n$  is
called \textit{Fredholm} if the homology spaces $H_p(S,H)$ of its Koszul complex are all finite dimensional. In
this case $\overline{H}_p(S,H) = H_p(S,H)$ for all $p$ and
the \textit{Fredholm index} of $S$ is defined as ${\rm ind}(S) = \sum_{0 \leq p \leq n} (-1)^p \dim H_p(S,H).$

\begin{corollary} \label{index}
Let $T \in B(H)^n$ be a finitely generated $\gamma$-graded commuting row contraction with a weak resolution
$(H^2_n(\mathcal D^{\bullet}),\varphi^{\bullet})$. Then there are vector space isomorphisms
\[
\overline{H}_p(T,H) \cong H_p(\mathcal D^{\bullet},\varphi^{\bullet}(0)) \cong \mathcal D^p \quad (p \geq 0).
\]
If $T$ is Fredholm, its index is given by ${\rm ind}(T) = \sum_{0 \leq p \leq n} (-1)^p \dim \mathcal D^p$.
\end{corollary} 

\begin{proof}
It remains to show that $H_p(\mathcal D^{\bullet},\varphi^{\bullet}(0)) \cong \mathcal D^p$ for $p \geq 0$.
By construction the complex $(\mathbb C[z] \otimes \mathcal D^{\bullet},\varphi^{\bullet})$ satisfies the
minimality condition
\[
\varphi^p (\mathbb C[z] \otimes \mathcal D^p) \subset \sum^n_{i=1} z_i (\mathbb C[z] \otimes \mathcal D^{p-1}) \quad (p \geq 1).
\]
Hence $\varphi^p(0) = 0$ for $p \geq 0$ and the assertion follows.
\qedhere
\end{proof}

\begin{remark}
{\rm (a) The above double complex methods show in particular that
\[
\mathcal D^p \cong H_p(\mathcal D^{\bullet},\varphi^{\bullet}(0)) \cong H_p(z,\tilde{H}) = 0 \quad \mbox{for } p > n.
\]
This argument could be used to replace the reference to Hilbert's Syzygy Theorem in the proof of Theorem \ref{resolutions}.\\

(b) Using dilation results due to Arveson \cite{AIII} or M\"uller and Vasilescu \cite{MV} one can construct, for each pure commuting row contraction 
$T \in B(H)^n$, an infinite resolution
\[
\ldots \stackrel{\varphi^2}{\longrightarrow} H^2_n(\mathcal E^1) \stackrel{\varphi^1}{\longrightarrow} H^2_n(\mathcal E^0)
\longrightarrow H \rightarrow 0
\]
intertwining the tuples $M_z \in B(H^2_n(\mathcal E^i))^n$ and $T \in B(H)^n$, where the spaces $\mathcal E^i$
are typically infinite dimensional and the maps $\varphi^i: \mathbb B \rightarrow B(\mathcal E^i,\mathcal E^{i-1})$
are partially isometric multipliers. As in the proof of Corollary \ref{index} one can use the double complex 
$K = (K_{\bullet}(z-\lambda,H^2_n(\mathcal E^{\bullet})))$ to show that there are vector space isomorphisms
\[
H_p(T-\lambda,H) \cong H_p(\mathcal E^{\bullet},\varphi^{\bullet}(\lambda))
\]
for $\lambda \in \mathbb B$ and $p \geq 0$. This method was used by Greene in \cite{G} to calculate  
the homology groups of the Koszul complex of $T-\lambda$.	
}
\end{remark}

Standard results from the dilation theory for row contractions \cite{AIII,MV} can be used to give a more explicit definition of the maps $\Pi^i:H^2_n(\mathcal D^i)\rightarrow H^2_n(\mathcal D^{i-1})$
 in the setting of Theorem \ref{resolutions}. Let $S \in B(K)^n$ be a pure commuting row contraction on a Hilbert space $K$. The defect operators of $S$ are defined as
\[
D_S = (1_{K^n}-S^\ast S)^{1/2}\in B(K^n),\ D_{S^\ast}=(1_K-SS^\ast)^{1/2}\in B(K),
\]
where $S:K^n\rightarrow K,(x_i)\mapsto \sum_{1\leq i\leq n}S_ix_i$, is the row operator induced by $S$
and $S^\ast:K\rightarrow K^n,x\mapsto (S^\ast_ix)_{1\leq i\leq n}$, denotes its adjoint. Define
$\mathcal D_S  =\overline{{\rm Im}(D_S)} \subset K^n$ and $\mathcal D_{S^\ast}=\overline{{\rm Im}(D_{S^\ast})}\subset K$. Then
\[
j_S: K\rightarrow H^2_n(\mathcal D_{S^\ast}), \; 
j_S(x)(z)=\sum_{\alpha\in \mathbb N^n}\frac{|\alpha|!}{\alpha!}(D_{S^\ast}S^{\ast\alpha}x)z^\alpha
\]
is an isometry that intertwines $S^\ast \in B(K)^n$ and $M^\ast_z \in B(H^2_n(\mathcal D_{S^\ast}))^n$. 
For $z \in \mathbb B$, define $Z:K^n \rightarrow K, (x_i)\mapsto \sum_{1\leq i\leq n}z_ix_i$. The \textit{characteristic function} of $S$
\[
\theta_S:\mathbb B\rightarrow B(\mathcal D_S,\mathcal D_{S^\ast}),\theta_S(z) = -S + D_{S^\ast}(1_K-ZS^\ast)^{-1}ZD_S
\]
induces a partially isometric multiplication operator
\[
M_{\theta_S}:H^2_n(\mathcal D_S)\rightarrow H^2_n(\mathcal D_{S^\ast}),f \mapsto \theta_Sf
\]
such that $j_S j_S^* + M_{\theta_S}M^*_{\theta_S} = 1_{H^2_n(\mathcal D_{S^*})}$.

Note that ${\rm Im}\ M_{\theta_S} \subset \{f \in H^2_n(\mathcal D_{S^*}); f(0) = 0 \}$ if and only if
$S|\mathcal D_S = 0$. This is easily seen to happen if and only if $S$ is a partial isometry. In this case,
\[
1 - S^*S = P_{{\rm Ker} \ S} = D_S \quad \mbox{and} \quad \mathcal D_S = {\rm Ker} \ S.
\]

Let $\mathcal D$ be a Hilbert space and $M\subset H^2_n(\mathcal D)$ a closed $M_z$-invariant subspace.
The compression $S = P_KM_z|K \in B(K)^n$ of $M_z$ to the co-invariant subspace $K = H^2_n(\mathcal D) \ominus M$ is
a pure commuting row contraction. Suppose in addition that $M \subset \{f \in H^2_n(\mathcal D); f(0) = 0 \}$. Then for
$x \in \mathcal D$ and $f \in M$, we obtain
\[
\langle f,x \rangle_{H^2_n(\mathcal D)} = \langle f(0),x \rangle_{\mathcal D} = 0.
\]
Hence
\[
D_{S^*} = (P_K (1_{H^2_n(\mathcal D)} - \sum^n_{i=1} M_{z_i} M^*_{z_i})|K)^{1/2} = (P_K P_{\mathcal D}|K)^{1/2} 
= P^K_{\mathcal D} \in B(K)
\]
is the orthogonal projection of $K$ onto $\mathcal D \subset M^{\bot} = K$ and $\mathcal D_{S^*} = \mathcal D$.
An elementary calculation shows that $j_S$ acts as the inclusion map
$j_S: K \rightarrow H^2_n(\mathcal D), f \mapsto f$. Hence ${\rm Im} \ M_{\theta_S} = M \subset 
\{f \in H^2_n(\mathcal D); f(0) = 0 \}$. Consequently $D_S = P_{{\rm Ker}\ S}$ and $\mathcal D_S = {\rm Ker}\ S$. 
By Lemma 9 and Lemma 11 in \cite{E} the restriction of $M_{\theta_S}$ to the closed subspace
\[
\tilde{\mathcal D} = \{ x \in \mathcal D_S; \; D_S x \in {\rm Ker}(\delta^{S^*}_{n-1}) \} = 
{\rm Ker}\ S \cap {\rm Ker}(\delta^{S^*}_{n-1})
\]
yields a unitary operator
\[
\tilde{\mathcal D} \rightarrow W_{M_z}(M), \; x \mapsto \theta_S x.
\]
Since the complex $K_{\bullet}(M^*_z,H^2_n(\mathcal D))$ is exact in every degree $p < n$ and is dual to the
cochain Koszul complex $K^{\bullet}(M_z,H^2_n(\mathcal D))$ (Section 2.6 in \cite{EP}), we find that
\[
{\rm Ker}\ S \cap {\rm Ker}(\delta^{S^*}_{n-1}) = {\rm Ker}\ S \cap {\rm Ker}(\delta^{M_z^*}_{n-1}) = 
{\rm Ker}\ S \cap ({\rm Ker}\ M_z)^{\bot}.
\]
%where $S$ and $M_z$ denote as before the row operators $S: K^n \rightarrow K$ and $M_z: H^2_n(\mathcal D)^n \rightarrow H^2_n(\mathcal D)$.
By the proof of Theorem 8 in \cite{E} (with $m = 1$), for $x \in \mathcal D_S$, the function $\theta_S x \in H^2_n(\mathcal D)$
takes the form
\[
\theta_S x = M_z (\oplus j) D_S x = M_z x.
\]
Thus we have shown that the row operator $M_z$ induces a unitary operator
\[
{\rm Ker}\ S \cap ({\rm Ker}\ M_z)^{\bot} \rightarrow W_{M_z}(M), x \mapsto M_z x.
\]

Let $T \in B(H)^n$ be as in Theorem \ref{resolutions} a finitely generated $\gamma$-graded commuting row contraction. If in
the construction of a weak resolution for $T$ mappings
\[
H^2_n(\mathcal D^{i-1}) \stackrel{\Pi^{i-1}}{\longrightarrow} H^2_n(\mathcal D^{i-2}) \stackrel{\Pi^{i-2}}{\longrightarrow}
\ldots \stackrel{\Pi^1}{\longrightarrow} H^2_n(\mathcal D^0) \stackrel{\Pi^0}{\longrightarrow} H \rightarrow 0  
\]
have been defined, as in the section leading to Theorem \ref{resolutions}, then by applying the preceding 
observations to the space $M^{i-1} = {\rm Ker}(\Pi^{i-1}) \subset H^2_n(\mathcal D^{i-1})$ and to the compression
\[
S^i = P_{K^{i-1}} M_z|K^{i-1} \in B(K^{i-1})^n
\]
of $M_z$ onto $K^{i-1} = H^2_n(\mathcal D^{i-1}) \ominus M^{i-1}$ one obtains a finite-dimensional subspace
\[
\mathcal D^i = \tilde{\mathcal D} = {\rm Ker}\ S^i \cap ({\rm Ker}\ M_z^{\mathcal D^{i-1}})^{\bot} \subset {\rm Ker}\ S^i = \mathcal D_{S^i} 
\]
such that the characteristic function $\theta_{S^i}$ of $S^i$ induces a unitary operator
\[
\mathcal D^i \rightarrow W_{M_z}(M^{i-1}), \; x \mapsto \theta_{S^i} x = M_z^{\mathcal D^{i-1}} x.
\]
An inductive application of this construction gives rise to a weak resolution 
\[
0 \rightarrow H^2_n(\mathcal D^n) \stackrel{\theta_{S^n}}{\longrightarrow} \ldots \stackrel{\theta_{S^2}}{\longrightarrow} 
H^2_n(\mathcal D^1) \stackrel{\theta_{S^1}}{\longrightarrow} H^2_n(\mathcal D^0)
\stackrel{\Pi^0}{\longrightarrow} H \rightarrow 0
\]
of $T$. Restriction to the polynomials yields the minimal graded free resolution of $\tilde{H}$
\[
0 \rightarrow \mathbb C[z] \otimes \mathcal D^n \stackrel{\varphi^n}{\longrightarrow} \ldots \stackrel{\varphi^2}{\longrightarrow} 
\mathbb C[z] \otimes \mathcal D^1 \stackrel{\varphi^1}{\longrightarrow} \mathbb C[z] \otimes\mathcal D^0
\stackrel{\Pi^0}{\longrightarrow} \tilde{H} \rightarrow 0,
\]
where the maps $\varphi^i$ are the $\mathbb C[z]$-module homomorphisms induced by multiplication with the operator-valued
polynomials
\[
\varphi^i: \mathbb B \rightarrow B(\mathcal D^i, \mathcal D^{i-1}), \varphi^i(z)x = \theta_{S^i}(z)(x) = (M_z^{\mathcal D^{i-1}}x)(z).
\]

\vspace{1.5cm}

J. Eschmeier \\
Fachrichtung Mathematik\\
Universit\"at des Saarlandes\\
Postfach 151150\\
D-66041 Saarbr\"ucken, Germany\\
email: eschmei@math.uni-sb.de


\begin{thebibliography}{99}

\vspace{.7cm}

\bibitem{ABFP}
C. Apostol, H. Bercovici, C. Foias and C. Pearcy, {\em Invariant subspaces, dilation theory, and the structure of 
the predual of a dual algebra},
J. Funct. Anal. 63 (1985), 369-404.

\bibitem{ARS}
A. Aleman, S. Richter and C. Sundberg, {\em Beurlings's theorem for the Bergman space},
Acta Math. 177 (1996), 275-310.

\bibitem{AIII}
W. Arveson, {\em Subalgebras of C*-algebras III, Multivariable operator theory},
Acta Math. 181 (1998), 159-228.

\bibitem{A}
W. Arveson, {\em The free cover of a row contraction},
Documenta Math. (2004), 137-161.

\bibitem{BEKS}
M. Bhattacharjee, J. Eschmeier, D. K. Keshari and J. Sarkar, {\em Dilations, wandering subspaces, and inner functions},
Linear Algebra Appl. 523 (2017), 263-280.

\bibitem{B}
C. Barbian, {\em A characterization of multiplication opertors on reproducing kernel Hilbert spaces},
J. Operator Theory (2011), 235-240.

\bibitem{DM}
R. G. Douglas, G. Misra, {\em Quasifree resolutions of Hilbert modules},
Integral Equations Operator Theory (2003), 435-456.

\bibitem{Comp}
J. Eschmeier, {\em Grothendieck's comparison theorem and multivariable Fredholm theory},
Arch. Math. 92 (2009), 461-475.

\bibitem{E}
J. Eschmeier, {\em Bergman inner functions and $m$-hypercontractions},
J. Funct. Anal. 275 (2018), 73-102.

\bibitem{EP}
J. Eschmeier, M. Putinar, {\em Spectral decompositions and analytic sheaves},
London Mathematical Society Monographs, New Series, {\bf 10}, Clarendon Press, Oxford, 1996.

\bibitem{G}
D. Greene, {\em On free resolutions in multivariable operator theory},
J. Funct. Anal. 200 (2003), 429-450.

\bibitem{H}
P. Halmos, {\em Shifts on Hilbert spaces},
J. Reine Angew. Math. 208 (1961), 102-112.

\bibitem{HZ}
H. Hedenmalm, K. Zhu, {\em On the failure of optimal factorization for certain weighted Bergman spaces},
Complex Variables 19 (1992), 165-176.

\bibitem{HRS}
H. Hedenmalm, S. Richter and K. Seip {\em Interpolation sequences and invariant subspaces of arbitrary
index in the Bergman spaces},
J. Reine Angew. Math. 477 (1996), 13-30.

\bibitem{L}
S. Lang, {\em Algebra. Second Edition},
Addison-Wesley, New York, 1984.

\bibitem{MV}
V. M\"uller, F.-H. Vasilescu, {\em Standard models for commuting multioperators},
Proc. Amer. Math. Soc. 117 (1993), 979-989.

\bibitem{N}
D. G. Northcott, {\em Lessons on rings, modules an multiplicities},
Cambridge University Press, London, 1968.

\bibitem{Sh}
S. Shimorin, {\em On Beurling-type theorems in weighted $\ell^2$ and Bergman spaces},
Proc. Math. Amer. Soc. 131 (2002), 1777-1787.

\bibitem{ZS}
O. Zariski, P. Samuel, {\em Commutative algebra, Volume II,},
Springer, Berlin, 1960.


\end{thebibliography}
\end{document}